\documentclass{amsart}
\usepackage{amssymb}
\usepackage{amscd}
\usepackage{verbatim}
\usepackage{epsfig}

\newcommand{\pa}{\partial}
\newcommand{\CI}{C^\infty}

\newcommand{\Op}{\operatorname{Op}}

\newcommand{\semi}{\hbar}

\newcommand{\RR}{\mathbb{R}}
\newcommand{\Cx}{\mathbb{C}}

\newcommand{\sphere}{\mathbb{S}}

\newcommand{\cL}{\mathcal L}

\newcommand{\ff}{{\mathrm{ff}}}

\newcommand{\be}[1]{\begin{equation}\label{#1}}
\newcommand{\ee}{\end{equation}}

\newtheorem{theorem}{Theorem}[section]

\newtheorem{proposition}[theorem]{Proposition}
\newtheorem{corollary}[theorem]{Corollary}

\theoremstyle{definition}
\newtheorem{definition}[theorem]{Definition}

\theoremstyle{remark}

\numberwithin{equation}{section}

\begin{document}

\title[Semiclassical and standard algebras]{On the relationship between the semiclassical and standard
  pseudodifferential algebras}

\author[Andras Vasy]{Andr\'as Vasy}
\thanks{The author gratefully acknowledges support from the National
  Science Foundation under grant numbers DMS-1953987 and
  DMS-2247004. The author is also very grateful to Oliver Petersen and
Jared Wunsch for comments on the manuscript, and to the anonymous
referees for the careful reading of the manuscript and the comments
which greatly improved the exposition.}

\dedicatory{To the memory of Steve Zelditch, his mathematics and his infectious
  enthusiasm}

\address{Department of Mathematics, Stanford University, Stanford, CA
94305-2125, U.S.A.}
\email{andras@math.stanford.edu}


\begin{abstract}
  In this short paper we discuss the precise relationship between the
semiclassical and standard pseudodifferential algebras and explore
implications such as for large spectral parameter elliptic
estimates, even in the case of pseudodifferential spectral familes.
We also explain the connection between second
microlocalization and this relationship.
\end{abstract}

\maketitle

\section{Introduction}
The purpose of this paper is to discuss the precise relationship
between the semiclassical and standard pseudodifferential operator
algebras and explain the connection to second microlocalization
discussed in \cite{Bony:Second, Vasy-Wunsch:Semiclassical,
  Vasy:Zero-energy, Vasy:Limiting-absorption-lag, Vasy:Zero-energy-lag}. One
immediate consequence is large spectral parameter resolvent estimates for elliptic
{\em pseudodifferential} operators  (as opposed to arguments which are
specific to the standard
differential case), and in addition a very simple approach to the
functional calculus.

In order to state the main theorem, we need to
introduce some notation that will be
explained in detail below. First, there is flexibility in the overall
context; for simplicity and definiteness we usually refer to the
bounded geometry setting, thus including compact manifolds without
boundary in particular, but for instance the same kind of arguments
work in Melrose's scattering pseudodifferential algebra
\cite{RBMSpec}. With this in mind, $\Psi_\infty^{m,k}$ denotes the
space of pseudodifferential operators depending on a parameter, which
we consider small (say, in $[0,1)$) and denote by $h$; here the
differential order is $m$ and parameter order is $k$. Note that here
$h$ is purely a parameter on which symbols depend (smoothly or in a
uniformly bounded way), distingushing this from the semiclassical
pseudodifferential operator algebra. Next,
$\Psi_{\infty,\semi}^{m,l,k}$ denotes\footnote{Below we use $h$ for
  the actual semiclassical parameter, and $\hbar$ as a subscript to
  denote semiclassical objects.}
the combined semiclassical-classical algebra in this setting, with
semiclassical differential order $m$, semiclassical parameter order
$l$ and classical parameter order $k$. Then we have the following
result:

\begin{theorem}\label{theorem:main}
Suppose that $A$ is an elliptic pseudodifferential operator in
$\Psi_\infty^{m,0}$, $m>0$, with
real principal symbol. Suppose $\lambda$ is in a compact subset of $\Cx$ disjoint
from $\RR$. Then $A-\lambda/h^m\in\Psi_{\infty,\semi}^{m,m,m}$ and is
elliptic there, invertible for $h$ sufficiently small, with
$$
(A-\lambda/h^m)^{-1}\in\Psi_{\infty,\semi}^{-m,-m,-m}=\Psi_\infty^{-m,0}\cap\Psi_\infty^{0,-m}.
$$
\end{theorem}

We remark here that the issue for the large parameter spectral family for pseudodifferential operators
$A\in\Psi_\infty^{m,0}$, $m>0$, is that it does not lie in Shubin's
class of
large parameter pseudodifferential operators \cite{Shubin:Pseudo} since an expression like
$a(z,\zeta)-\sigma^m$ (where $\sigma$ is the large parameter) is not
jointly symbolic in $(\zeta,\sigma)$, unless $a$ is a polynomial,
i.e. $A$ is a differential operator.

This theorem has standard generalizations as long as the symbol takes values
in a closed cone not containing $\lambda$. Theorem~\ref{theorem:main}
moreover provides a very simple approach to complex powers
for self-adjoint $A$ as the
$(A-\lambda/h^m)^{-1}\in\Psi_\infty^{-m,0}$ statement indicates: along
a contour, one
integrates a family of {\em standard} pseudodifferential operators
with uniform behavior in this class. 
In fact, after the preparation of the initial version of
this paper, non-elliptic versions of the method (proceeding
in the same setting, but via propagation estimates) have been used to
analyze the spectral family of Lorentzian Dirac operators by Dang,
Vasy and Wrochna \cite{Dang-Vasy-Wrochna:Dirac} and this in turn has been utilized
to construct complex powers even in this non-self-adjoint setting,
extending the scalar (thus self-adjoint) Lorentzian work of Dang and
Wrochna \cite{Dang-Wrochna:Complex} which relied on the author's
earlier less precise results \cite{Vasy:Self-adjoint}.

We now proceed to define these classes.
Recall that standard pseudodifferential operators on $\RR^n$ are of
the form
$$
\Op(a)u(z)=(2\pi)^{-n}\int e^{i(z-z')\cdot\zeta}a(z,\zeta)u(z')\,dz'\,d\zeta,\qquad\Op(a)\in\Psi^m_\infty,
$$
where $u$ is, say, Schwartz, $a\in S^m_\infty$ a symbol, and the integral is interpreted as an
oscillatory integral. Here the symbol class demands
$$
|(D_z^\alpha D_\zeta^\beta a)(z,\zeta)|\leq C_{\alpha\beta} \langle\zeta\rangle^{m-|\beta|};
$$
we use this `bounded geometry on $\RR^n$' type class as an example for
simplicity and as various versions are equivalent on compact
manifolds. The subscript $\infty$ is inserted as a reminder that this is
H\"ormander's uniform symbol class (in $z$), i.e.\ $z$ is not a
symbolic variable, though again the discussion simply extends to the
scattering class of Melrose, \cite{RBMSpec}, for instance, which in
this $\RR^n$ case first arose in the works of Shubin
\cite{Shubin:Pseudodifferential} and Parenti \cite{Parenti:Operatori}.
Of course we may consider a family of operators depending on a
parameter $h\in[0,1]$, so we have a family of symbols $a=(a_h)$ with
uniform estimates:
$$
|(D_z^\alpha D_\zeta^\beta a)(z,\zeta,h)|\leq C_{\alpha\beta} \langle\zeta\rangle^{m-|\beta|};
$$
this gives the class $\Psi_\infty^{m,0}$ in the above notation; if the
right hand instead has $C_{\alpha\beta} h^{-k}
\langle\zeta\rangle^{m-|\beta|}$, which we denote by $S_\infty^{m,k}$,
we obtain the class
$\Psi_\infty^{m,k}$. (The choice of the orders is such that the spaces
become larger with each index.)

On the other hand, corresponding to a {\em family} $a=(a_h)$, semiclassical pseudodifferential operators on $\RR^n$ are of
the form
$$
\Op_\semi(a_h)u(z)=(2\pi h)^{-n}\int
e^{i(z-z')\cdot\zeta_\semi/h}a(z,\zeta_\semi,h)u(z')\,dz'\,d\zeta,\qquad \Op_\semi(a_h)\in\Psi_\semi^m,
$$
where $u$ is again, say, Schwartz, $a\in S^m_\infty$, and this
expression is considered for $h\in(0,1]$. More generally, when $a\in
S^{m,k}_\infty$, we define $\Op_{\semi}(a_h)\in\Psi_\semi^{m,k}$ this way. Of course, for $h>0$, a
simple change of variables connects these two:
$$
\Op_\semi(a_h)u(z)=(2\pi)^{-n}\int
e^{i(z-z')\cdot\zeta}a(z,h\zeta,h)u(z')\,dz'\,d\zeta=\Op(\tilde a_h)u(z),
$$
where $\tilde a(z,\zeta,h)=a(z,h\zeta,h)$.

This is an explicit, if
singular in the limit $h\to 0$, connection between the two classes,
but it is useful to interpret this more geometrically. This has a
particular virtue if one considers a semiclassical family of operators but would
like to have it act on non-semiclassical function spaces. A need for
this arises, for instance, in the holomorphic functional calculus, in
which functions of a, say, elliptic, positive operator (on say a
compact manifold) are expressed as a contour integral involving the
resolvent: the large spectral parameter behavior (as we go out to
infinity along the contour) of the resolvent can be interpreted as a semiclassical
problem, but if we want to get at mapping properties of the result
directly, we need to work on a fixed function space (not on an
$h$-dependent one).

We remark here that although
we wrote the above formulae for $\RR^n$, they immediately transfer
essentially completely to
compact manifolds via localization; of course due to far-from diagonal
smoothing contributions, one needs to say a bit more in that case,
namely add smooth (in the manifold variables) Schwartz kernels, with uniform in $h$ bounds on all
derivatives, explicitly to $\Psi_\infty^{m,0}$, and similarly for the
semiclassical case one needs to add smooth (in the manifold variables) Schwartz kernels that are
rapidly decaying in $h$. Indeed, in order to simplify the geometric
discussion (avoiding the need to discuss the fibration over
$[0,1)_h$ in this context, and indeed b-fibration later after a blow-up, thus also
avoiding the need to add
tangency to its fibers conditions to the vector fields) it is convenient below to strengthen
our classes; from now on we demand conormal to $h=0$ behavior for our
symbol classes, i.e.\ we strengthen our symbol estimates to
\begin{equation}\label{eq:conormal-in-h-symbols}
|((hD_h)^\gamma D_z^\alpha D_\zeta^\beta a)(z,\zeta,h)|\leq C_{\alpha\beta\gamma} h^{-k}\langle\zeta\rangle^{m-|\beta|}.
\end{equation}

For our purpose, it is helpful to adopt a (partially) compactified
perspective due to Melrose \cite{RBMSpec} in which the fibers of the cotangent bundle $(\RR^n)^*_\zeta$
are compactified to balls $\overline{\RR^n}$, with boundary defining
function $\rho_\infty=|\zeta|^{-1}$ (more precisely,
$\langle\zeta\rangle^{-1}$; the two are equivalent outside a compact
subset of $\RR^n$):
$\overline{T^*}\RR^n=\RR^n_z\times\overline{(\RR^n)^*_\zeta}$; see
also \cite{Vasy:Minicourse}. More concretely, as $\{\zeta\in(\RR^n)^*:\
|\zeta|>1\}$ can be identified with $(1,\infty)_{|\zeta|}\times\sphere^{n-1}$
via `polar coordinates',
in this compactification a boundary at infinity is added by
identifying this in turn with
$(0,1)_{|\zeta|^{-1}}\times\sphere^{n-1}$, and regarding the latter as
a subset of $[0,1)_{\rho_\infty}\times\sphere^{n-1}$, where now the
ideal boundary $\{\rho_\infty=0\}$ has been added. The
symbol estimates (locally in $z$, which suffices for compact manifold
applications) are equivalent to the requirement that for all $N$ and
vector fields $V_1,\ldots,V_N$ tangent to $\pa \overline{T^*}\RR^n$
there is $C>0$ such that
$$
|V_1\ldots V_N a|\leq C\rho_\infty^{-m}.
$$
Indeed, if each $V_k$ is either $D_{\zeta_j}$ or $\zeta_i D_{\zeta_j}$
or $D_{z_j}$ for some
$i,j$, with $D_{\zeta_j}$ relevant only near $\zeta=0$, the estimates
are easily seen to be equivalent to symbol estimates, and these $V_k$
span, over $\CI(\overline{T^*}\RR^n)$, the space of smooth vector
fields tangent to $\pa \overline{T^*}\RR^n$.
As it is the fibers that are compactified in this manner, we name the
boundary, $\{\rho_\infty=0\}$, fiber infinity.
Adding a parameter $h$, we obtain the parameter dependent partially
compactified cotangent bundle
$$
\overline{T^*}\RR^n\times[0,1]_h=\RR^n_z\times\overline{(\RR^n)^*_\zeta}\times[0,1]_h,
$$
with a uniform in $h$ version of the membership statement;
$\{\rho_\infty=0\}$ is still called fiber infinity. As already
mentioned prior to \eqref{eq:conormal-in-h-symbols}, it is
actually helpful to consider the conormal parameter dependent version
of the class, i.e.\ when
 for all $N$ and
vector fields $V_1,\ldots,V_N$ tangent to $\pa(
\overline{T^*}\RR^n\times[0,1))$ (we suppress the uninteresting $h=1$
boundary by explicitly removing it)
there is $C>0$ such that
$$
|V_1\ldots V_N a|\leq C\rho_\infty^{-m}h^{-k};
$$
these vector fields include $hD_h$, and are spanned by $hD_h$ together
with the ones listed above: $D_{z_j}$, $D_{\zeta_j}$, $\zeta_i D_{\zeta_j}$.

\begin{figure}[ht]
\begin{center}
\includegraphics[width=95mm]{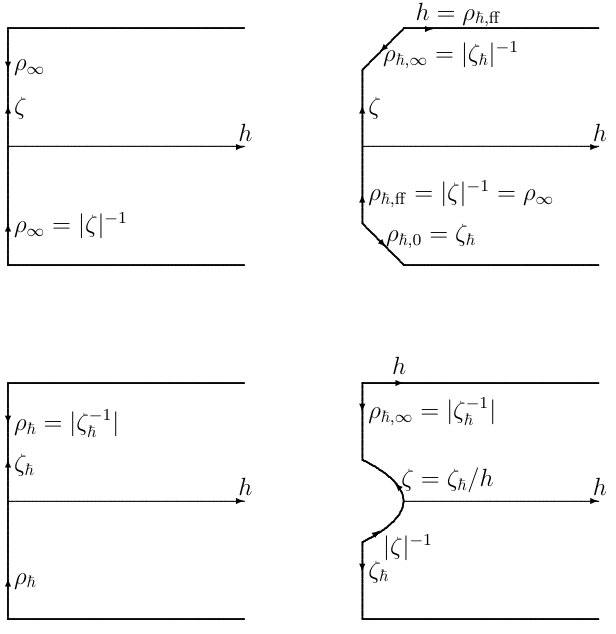}
\end{center}
\caption{The mixed semiclassical-classical symbol space, illustrated
  for one dimensional underlying space, thus cotangent bubdle fibers. The top row
  shows the blow up of the parameter ($h$) dependent fiber compactified standard cotangent
  bundle at the corner; the bottom row shows the blow up of the fiber
  compactified semiclassical cotangent bundle at the zero section at
  $h=0$. The resulting resolved spaces on the right are naturally
  diffeomorphic as indicated. The original unresolved spaces are on
  the left. The base manifold direction $z$ is not shown; it can be
  thought to be pointing out of the page. The reader should keep in
  mind that while various defining functions are global, the
  coordinate expressions for them are local, see e.g.\
  $\rho_{\semi,\ff}$ on the top right. Also, {\em for each picture} the top and bottom
  boundary hypersurfaces both correspond to the {\em same} fiber
  infinity; they appear distinct as the 0-sphere (the manifold being
  one dimensional for illustration) is disconnected. Similarly, on the
top right the two diagonal boundary hypersurfaces and on the bottom
right the two vertical boundary hypersurfaces are the same; these are
also identified between the top and the bottom picture.}
\label{fig:cl-to-semicl}
\end{figure}

We now proceed to introduce the mixed semiclassical-classical
pseudodifferential operators; see Figure~\ref{fig:cl-to-semicl} for an
illustration of the phase spaces.
Blowing up the corner, $\{h=0,\ \rho_\infty=0\}$, of
$\overline{T^*}\RR^n\times[0,1)$, to obtain
$$
[\overline{T^*}\RR^n\times[0,1);\pa \overline{T^*}\RR^n\times\{0\}]
$$
introduces a new
front face $\ff$. Away from the lift of $h=0$, i.e\ where $h$ is
relatively large relative to $\rho_\infty$, i.e.\ $\rho_\infty<Ch$, projective coordinates
are given by $h$, $\rho_\infty/h$ as well as tangential variables to
the corner, while away from the lift of $\rho_\infty=0$, i.e.\ where
$\rho_\infty$ is relatively large relative to $h$, i.e.\ $h<C\rho_\infty$, projective
coordinates are given by $\rho_\infty,h/\rho_\infty$ as well as tangential variables to
the corner. But as observed above $\rho_\infty=|\zeta|^{-1}$, so
\begin{equation}
  \label{eq:semi-fiber-infinity-local}
\rho_\infty/h=|h\zeta|^{-1}=|\zeta_\semi|^{-1}.
\end{equation}
Correspondingly near the lift of fiber infinity, i.e.\
$\rho_\infty=0$, this space is just the same as fiber infinity for the
semiclassical space, so we may call it `semiclassical fiber infinity'
(and the earlier `classical fiber infinity' for clarity)!

A different way of arriving at the same result is to
consider the blow-up of the zero section at $h=0$
of the semiclassical cotangent bundle; Figure~\ref{fig:cl-to-semicl}
illustrates why the resulting spaces are naturally diffeomorphic
(i.e.\ the natural identification in the interior extends to a
diffeomorphism). This second approach is called second
microlocalization, in this case for the semiclassical symbols at the
zero section.

Returning to the corner blow up perspective, we write $\rho_{\semi,\infty}$
for the defining function of the lift of classical fiber infinity (i.e.\
semiclassical fiber infinity), which is (equivalent to)
\eqref{eq:semi-fiber-infinity-local} locally near the lift of fiber
infinity; globally one can for instance
take it to be\footnote{An equivalent choice is $\rho_{\semi,\infty}=\frac{\rho_\infty}{(h^2+\rho_\infty^2)^{1/2}}=\frac{\langle\zeta\rangle^{-1}}{(h^2+(1+|\zeta|^2)^{-1})^{1/2}}$, natural
    from the corner blow up perspective.}
$$
\rho_{\semi,\infty}=\langle\zeta_\semi\rangle^{-1}=(1+|\zeta_\semi|^2)^{-1/2}=(1+h^2|\zeta|^2)^{-1/2}
$$
(which is immediate from the second approach). Next, we write $\rho_{\semi,\ff}$ for
the defining function of the front face, which is $h$ near the lift of
fiber infinity and $\rho_\infty$ near the lift of $h=0$; globally we
can take it to be
$$
\rho_{\semi,\ff}=(h^2+\rho_\infty^2)^{1/2}=(h^2+(1+|\zeta|^{2})^{-1})^{1/2}=h(1+(h^2+|\zeta_\semi|^2)^{-1})^{1/2}
$$
as is immediate from the corner blow up perspective since $h=0$,
$\langle\zeta\rangle^{-1}=0$ is being blown up.
Finally, we write
$\rho_{\semi,0}$ for the defining function of the lift of $h=0$, which
is locally given (up to equivalence) by $|\zeta_{\semi}|$, and
globally by\footnote{An equivalent choice is
  $\rho_{\semi,0}=\frac{h}{(h^2+(1+|\zeta|^2)^{-1})^{1/2}}$, again natural
    from the corner blow up perspective.}
$$
\rho_{\semi,0}=\frac{(h^2+|\zeta_\semi|^2)^{1/2}}{(1+|\zeta_\semi|^2)^{1/2}};
$$
we call this the parameter boundary. Notice that, as is necessarily
the case,
$\rho_{\semi,0}\rho_{\semi,\ff}$ is equivalent to $h$ as
$$
\frac{(h^2+|\zeta_\semi|^2)^{1/2}}{(1+|\zeta_\semi|^2)^{1/2}}(1+(h^2+|\zeta_\semi|^2)^{-1})^{1/2}=
\frac{(h^2+|\zeta_\semi|^2+1)^{1/2}}{(1+|\zeta_\semi|^2)^{1/2}}
  $$
  and the $h^2$ in the numerator is irrelevant in view of the term $1$ there.
Similarly, $\rho_{\semi,\infty}\rho_{\semi,\ff}$ is equivalent to $\rho_\infty=\langle\zeta\rangle^{-1}$.

We write $S_\infty^{m,l,k}$ for the
symbol class that arises on this resolved space:

\begin{definition}
  The space $S_\infty^{m,l,k}$ consists of conormal functions on the
  resolved space
  $$
  [\overline{T^*}\RR^n\times[0,1);\pa \overline{T^*}\RR^n\times\{0\}],
  $$
  i.e.\
satisfy estimates which are
stable under iterated application of vector fields tangent to all
boundary hypersurfaces:
$$
|V_1\ldots V_Na|\leq C\rho_{\semi,\infty}^{-m}\rho_{\semi,\ff}^{-l}\rho_{\semi,0}^{-k}.
$$
Here it suffices to consider the $V$'s being $hD_h$, $D_{z_j}$,
$D_{\zeta_j}$, $\zeta_i D_{\zeta_j}$ since over smooth functions on $[\overline{T^*}\RR^n\times[0,1);\pa \overline{T^*}\RR^n\times\{0\}]$ these
span the space of vector fields tangent to all
boundary hypersurfaces.

We call $m$ the semiclassical differential
order, $l$  the semiclassical growth order, $k$ the
standard parameter growth order at
$h=0$ as a family of symbols, and we use the ordering differential,
semiclassical, parameter decay at $h=0$ for the superscripts.
\end{definition}

We already noted that (up to equivalence of boundary defining functions)
$$
\rho_\infty=\rho_{\semi,\infty}\rho_{\semi,\ff},\
h=\rho_{\semi,0}\rho_{\semi,\ff},
$$
and in fact correspondingly
\begin{equation}\label{eq:resolved-symbol-relation}
S^{m,k}_\infty=S^{m,m+k,k}_\infty
\end{equation}
under the natural pullback identification. This identification breaks
for the corresponding {\em one-step polyhomogeneous symbols, with
  smooth dependence on $h$} (these are often called classical symbols,
but we avoid this term due to our different use of the word `classical'), for which the corresponding left
hand side is merely included in the corresponding right hand side. The
reason is simple: for instance, if all orders are $0$, locally both
classes are simply smooth functions on
$$\overline{T^*}\RR^n\times[0,1),\ \text{resp.}
\ [\overline{T^*}\RR^n\times[0,1);\pa \overline{T^*}\RR^n\times\{0\}],
$$
and certainly these two spaces of smooth functions differ, unlike the
spaces of conormal functions with all orders $0$ on the two spaces.

This classical-semiclassical relationship is completely analogous to
the relationship between the b- and sc- cotangent bundles: the
sc-cotangent bundle arises by blowing up the corner of the
fiber-compactified b-cotangent bundle, and conversely the b-cotangent
bundle arises by blowing up the zero section of the sc-cotangent
bundle at the boundary. While these statements are symmetric, there is
a big difference: blowing up the corner does not change the regularity
properties with respect to vector fields tangent to the boundary
hypersurfaces, hence the sc-algebra is naturally included in the
b-algebra (in a non-classical manner). On the other hand, blowing up
the zero section at the boundary is highly singular with respect to
the quantization map (one obtains an ill-behaved symbol class) which is the
reason second microlocalization (which exactly adds b-algebra features to
the scattering one) is considered delicate; see \cite{Vasy:Zero-energy}. Thus, just as
in the b-sc setting one should think of the process in the less symbolic b-terms, by blowing up the corner, in our
semiclassical setting we should prefer blowing up the corner of the
standard fiber-compactified (family) cotangent bundle at $h=0$ since
this does not affect the symbolic properties; the standard family
algebra is less symbolic (as it is not symbolic in $h$) than the
semiclassical algebra. Note that in the semiclassical setting the more difficult, via second
microlocalization, approach was carried out by Wunsch and the author
in \cite{Vasy-Wunsch:Semiclassical}. The present work of course only
covers the model case of \cite{Vasy-Wunsch:Semiclassical}, namely
second microlocalization at the zero section; this was transferred in that
work to other Lagrangian submanifolds by conjugation by semiclassical Fourier
integral operators, and an analogous process would be required
here. (When doing this, there is substantial extra work if, as we do
in this paper in the context of `full ellipticity below', one wants
to precisely describe the `residual' behavior, which is global along the
Lagrangian; this aspect was not covered in
\cite{Vasy-Wunsch:Semiclassical} either.)

One advantage of the blow-up procedure for even $\Psi^{m,k}_\infty$ is that one has a more refined
notion of elliptic set and
wave front set which are now subsets of semiclassical fiber infinity and the front
face of $[\overline{T^*}\RR^n\times[0,1);\pa
\overline{T^*}\RR^n\times\{0\}]$. Namely:

\begin{definition}\label{def:resolved-elliptic-family}
  Let $A\in\Psi^{m,k}_\infty$ be the left quantization\footnote{Right, Weyl,
  etc., quantizations would work equally well.} of $a\in
S^{m,k}_\infty$ modulo $\Psi^{-\infty,k}_\infty$. We say that a point
$\alpha$ in semiclassical fiber infinity or the front
face of $[\overline{T^*}\RR^n\times[0,1);\pa
\overline{T^*}\RR^n\times\{0\}]$ is in the elliptic set of $A$ if
$\alpha$ has a neighborhood $O$ in $[\overline{T^*}\RR^n\times[0,1);\pa
\overline{T^*}\RR^n\times\{0\}]$ such that $|a|_O|$ is bounded below
by a constant positive multiple of
$$
\rho_\infty^{-m}h^{-k}=\rho_{\semi,\infty}^{-m}\rho_{\semi,\ff}^{-m-k}\rho_{\semi,0}^{-k}.
$$
The wave front set is defined similarly, asking for infinite order
vanishing of $a|_O$ at semiclassical fiber infinity and the front
face.
\end{definition}

This {\em refines} the standard elliptic and wave front sets of
the family which would lie in classical fiber infinity; the difference
is that classical fiber infinity at $h=0$ is replaced by the whole new
front face.

Another advantage of this blow-up procedure is that we can introduce a
collection of
pseudodifferential operators $\Psi_{\infty,\semi}^{m,l,k}$ that has an additional order based
on quantizing $S^{m,l,k}_{\infty,\semi}$ using {\em the same
  `standard' (i.e.\ non-semiclassical)
  quantization map, $\Op$}. Before stating the details, if the operator arises from a
standard symbol of order $m,k$, i.e.\ is in $\Psi_\infty^{m,k}$, one
necessarily has $l=m+k$ corresponding to \eqref{eq:resolved-symbol-relation}, and
indeed the converse relationship also holds for symbols so
\begin{equation}\label{eq:classical-to-combined-conv}
\Psi_\infty^{m,k}=\Psi_{\infty,\semi}^{m,m+k,k}.
\end{equation}
Thus, standard pseudodifferentials {\em are} in this joint algebra,
even though typically they are not, unlike differential operators, in
the regular semiclassical algebra. Note, however, that for the {\em
  one-step polyhomogeneous, with smooth dependence on $h$} subalgebra
of $\Psi_\infty^{m,k}$, and the analogous one-step polyhomogeneous
subalgebra of $\Psi_{\infty,\semi}^{m,m+k,k}$, the analogue of the equality
\eqref{eq:classical-to-combined-conv} does {\em not} hold, unlike in
the symbolic algebra; we only have the
inclusion of the corresponding left hand side in the corresponding
right hand side.

One way to give the precise
definition of $\Psi^{m,l,k}_{\infty,\semi}$ in general is via noting that
$$
S^{m,l,k}_{\infty,\semi}\subset S_{\infty}^{m+\max(l-(m+k),0),k}
$$
since if $l\leq m+k$ then $\rho_{\semi,\ff}^l\geq
C\rho_{\semi,\ff}^{m+k}$ hence
$$
\rho_{\semi,\infty}^{-m}\rho_{\semi,\ff}^{-l}\rho_{\semi,0}^{-k}\leq C\rho_{\semi,\infty}^{-m}\rho_{\semi,\ff}^{-(m+k)}\rho_{\semi,0}^{-k}=C\rho_\infty^{-m}h^{-k},
$$
while if $l\geq m+k$ then $\rho_{\semi,\infty}^{m}\geq C
\rho_{\semi,\infty}^{l-k}$, so
$$
\rho_{\semi,\infty}^{-m}\rho_{\semi,\ff}^{-l}\rho_{\semi,0}^{-k}\leq C\rho_{\semi,\infty}^{-(l-k)}\rho_{\semi,\ff}^{-l}\rho_{\semi,0}^{-k}=C\rho_\infty^{-(l-k)}h^{-k},
$$
and the derivative estimates are equivalent on the two spaces. Thus,
one can simply define:

\begin{definition}
  The space $\Psi^{m,l,k}_{\infty,\semi}$ of semiclassical-classical pseudodifferential operators
$$
\Psi^{m,l,k}_{\infty,\semi}\subset \Psi_{\infty}^{m+\max(l-(m+k),0),k}
$$
is the image of the quantization map $\Op$ on $S^{m,l,k}$, plus
elements of $\Psi^{-\infty,k}_\infty$.
\end{definition}

Notice that with this definition applied in the
compact manifold setting, Schwartz kernels which are conormal to $h=0$, with
$h^{-k}$ bounds, are automatically included in $\Psi_{\infty,\semi}^{m,l,k}$.

Membership in
this subspace $\Psi^{m,l,k}_{\infty,\semi}$ is characterized by purely symbolic properties,
namely finite order vanishing conditions within this class, concretely
of order $\max(l-(m+k),0)$ at semiclassical fiber infinity and
$m+k-l+\max(l-(m+k),0)=\max(m+k-l,0)$ at the front face. Therefore,
using the standard left and right reduction, hence adjoint and
composition, formulae,
$\Psi_{\semi,\infty}$ is easily seen to {\em form a tri-filtered
$*$-algebra.} 

In addition to \eqref{eq:classical-to-combined-conv},
directly from the symbol level, where the underlying space can be
identified as a blow up of the semiclassical phase space at the zero
section at $h=0$, we also have the inclusion
\begin{equation}\label{eq:semi-to-combined-conv}
\Psi_\semi^{m,k}\subset\Psi_{\infty,\semi}^{m,k,k};
\end{equation}
this is a proper inclusion since for instance any smooth Schwartz
kernel with smooth dependence on $h$ (and bounded with all
derivatives) lies in $\Psi_{\infty,\semi}^{-\infty,-\infty,0}$, but is
not a semiclassical pseudodifferential operator (composition is not
even commutative to leading order in $h$!).

A key point is that just as $\Psi^{m,k}_\infty$ only has a principal symbol at
fiber infinity, i.e.\ it captures $\Psi^{m,k}_\infty$ modulo
$\Psi^{m-1,k}_\infty$, not gaining any decay at $h=0$,
$\Psi_{\infty,\semi}^{m,l,k}$ inherits this principal symbol at
semiclassical fiber infinity and the front front face, so the
principal symbol captures it modulo $\Psi_{\infty,\semi}^{m-1,l-1,k}$
which is a gain at semiclassical fiber infinity and the front face but
not at the parameter boundary. Namely:

\begin{definition}
The principal symbol of $A\in \Psi_{\infty,\semi}^{m,l,k}$, written as
the left quantization\footnote{Again, right or Weyl quantizations
  could be used equally well.} of $a\in S_{\infty,\semi}^{m,l,k}$
modulo $\Psi_{\infty,\semi}^{-\infty,-\infty,k}$, is the class of $a$
in $S_{\infty,\semi}^{m,l,k}/ S_{\infty,\semi}^{m-1,l-1,k}$.
\end{definition}

At the classical face, i.e.\ the lift of $h=0$, elements of
$\Psi^{m,k,l}_{\infty,\semi}$ have, just like elements of $\Psi^{m,k}_\infty$,
an {\em operator valued} symbol. Namely, for $\Psi^{m,k}_\infty$, if the element is actually
smooth in $h$ (rather than conormal), this operator valued symbol, or
{\em normal operator}, is the
$h=0$ value of $h^{k}$ times the family. For one-step polyhomogeneous
symbols in $S_{\infty,\semi}^{m,l,k}$ this normal operator can still
be considered as the restriction (after multiplying by $h^k$) of the symbol family to the parameter
boundary and quantized in the standard (non-semiclassical) manner; for
the general class $\Psi_{\infty,\semi}^{m,l,k}$ it is best to simply
consider it as a representative in $\Psi_{\infty,\semi}^{\infty,l,k}$
modulo $\Psi_{\infty,\semi}^{\infty,l,k-1}$; the first order is
irrelevant here since semiclassical fiber infinity does not intersect
the parameter boundary.

{\em Ellipticity} corresponds to the ellipticity
at the principal symbol level, as in
Definition~\ref{def:resolved-elliptic-family}:

\begin{definition}
  Let $A\in\Psi^{m,l,k}_{\infty,\semi}$ and let $a\in
  S_{\infty,\semi}^{m,l,k}$ be a representative of its principal symbol.
We say that a point
$\alpha$ in semiclassical fiber infinity or the front
face of $[\overline{T^*}\RR^n\times[0,1);\pa
\overline{T^*}\RR^n\times\{0\}]$ is in the elliptic set of $A\in\Psi^{m,l,k}_{\infty,\semi}$ if
$\alpha$ has a neighborhood $O$ in $[\overline{T^*}\RR^n\times[0,1);\pa
\overline{T^*}\RR^n\times\{0\}]$ such that $|a|_O|$ is bounded below
by a constant positive multiple of
$$
\rho_\infty^{-m}h^{-k}=\rho_{\semi,\infty}^{-m}\rho_{\semi,\ff}^{-m-k}\rho_{\semi,0}^{-k}.
$$
\end{definition}

The {\em wave front set} is defined similarly, asking for infinite order
vanishing of $a|_O$, where $A$ now is, say, the left quantization of
$a$ modulo $\Psi^{m,l,k}_{\infty,\semi}$, at semiclassical fiber infinity and the front
face.

If $A\in\Psi_{\infty,\semi}^{m,l,k}$ is
elliptic then by the standard symbolic construction there is a
parametrix $\tilde B\in \Psi_{\infty,\semi}^{-m,-l,-k}$ such that
$$
\tilde E=\tilde B A-I,\ \tilde F=A\tilde B-I\in \Psi_{\infty,\semi}^{-\infty,-\infty,0}=\Psi_\infty^{-\infty,0},
$$
and $\tilde E=\tilde E(h)$ (and similarly $\tilde F$) then gives a
uniformly bounded family of operators between any (standard)
Sobolev spaces. In particular, if the underlying manifold is compact or the pseudodifferential
operators are symbolic also at infinity (in which case there would be
also order $-\infty$ in that sense, and the Sobolev spaces would have
arbitrary decay), such as the scattering algebra,
these operators $\tilde E,\tilde F$ are in addition compact, giving a uniform (in $h$) Fredholm
theory.

On the other hand {\em full ellipticity in addition} includes the invertibility
of the normal operator:

\begin{definition}
  We say that $A\in\Psi_{\infty,\semi}^{m,l,k}$ if
  {\em fully elliptic} if it is elliptic and there exists a representative
$A_0\in \Psi_{\infty,\semi}^{\infty,l,k}$ of
$A\in\Psi_{\infty,\semi}^{m,l,k}$ modulo
$\Psi_{\infty,\semi}^{\infty,l,k-1}$ and
an element $B_0$ of $\Psi_{\infty,\semi}^{\infty,-l,-k}$ such that
$B_0A_0=A_0B_0=I\in\Psi_{\infty,\semi}^{0,0,0}$.
\end{definition}

Full ellipticity guarantees invertibility of the family $A$ for
small $h$:

\begin{proposition}\label{prop:full-parametrix}
If $A\in\Psi_{\infty,\semi}^{m,l,k}$ is fully elliptic then there
exists  $B\in
\Psi_{\infty,\semi}^{-m,-l,-k}$ such that
$$
BA=I+E,\qquad E\in\Psi_{\infty,\semi}^{-\infty,-\infty,-\infty},
$$
and similarly $AB=I+F$, $F\in\Psi_{\infty,\semi}^{-\infty,-\infty,-\infty}$.
\end{proposition}

See below for an example involving spectral families of
positive order operators.

  \begin{proof}
    Suppose that $A\in\Psi_{\infty,\semi}^{m,l,k}$ is
fully elliptic. Then, as mentioned above, by the standard symbolic construction there is a
parametrix $\tilde B\in \Psi_{\infty,\semi}^{-m,-l,-k}$ such that
$$
\tilde E=\tilde B A-I,\ \tilde F=A\tilde B-I\in \Psi_{\infty,\semi}^{-\infty,-\infty,0}.
$$
Let $A_0,B_0$ be as above using the full ellipticity of $A$. Let
$$
B_1=\tilde B-\tilde E B_0\in \Psi_{\infty,\semi}^{-m,-l,-k};
$$
then
$$
B_1A=\tilde B A-\tilde E B_0 A_0+\tilde E B_0(A_0-A)=I+\tilde E-\tilde
E+\tilde E B_0(A_0-A)=I+E_1,
$$
with $E_1\in\Psi_{\infty,\semi}^{-\infty,-\infty,-1}$,
and a standard iteration and asymptotic summation improves this to $B\in
\Psi_{\infty,\semi}^{-m,-l,-k}$ such that
$$
BA=I+E,\qquad E\in\Psi_{\infty,\semi}^{-\infty,-\infty,-\infty},
$$
and similarly for $AB$. Equality of left and right parametrices
modulo $\Psi_{\infty,\semi}^{-\infty,-\infty,-\infty}$ follows as
usual.
\end{proof}

\begin{corollary}
If $A\in\Psi_{\infty,\semi}^{m,l,k}$ is fully elliptic then there is
$h_0>0$ such that $A=A_h$ is invertible for $h<h_0$.
\end{corollary}

\begin{proof}
  By Proposition~\ref{prop:full-parametrix}, $E$ is $O(h^\infty)$ on any reasonable space
(such as $L^\infty_h L^2$) hence invertibility of $I+E$ for small $h$
follows, and then
$$
(I+E)^{-1}=I-E+E(I+E)^{-1}E
$$
together with the regularizing property of $E$ shows that
$(I+E)^{-1}-I\in \Psi_{\infty,\semi}^{-\infty,-\infty,-\infty}$ as
well.
\end{proof}

It is simple to read off mapping properties
of the joint standard-semiclassical spaces
$H_{\infty,\semi}^{s,r,p}$, which are defined, e.g.\ if all the
indices are $\geq 0$, as the space of those $u\in L^\infty_h L^2$ such that
$Au\in L^\infty_h L^2$ for some (hence all) fully elliptic
$A\in\Psi_{\infty,\semi}^{s,r,p}$. (One can also use spaces that are
$L^2_hL^2$, which are often more natural!) Equivalently, using only
ellipticity (rather than full ellipticity),  if all the
indices are $\geq 0$ and $r\geq p$, $H_{\infty,\semi}^{s,r,p}$ is the
space of those $u\in h^p L^\infty_h L^2$ such that
$Au\in L^\infty_h L^2$ for some (hence all) fully elliptic
$A\in\Psi_{\infty,\semi}^{s,r,p}$.
Namely, the mapping property is
$$
\Psi_{\infty,\semi}^{m,l,k}\subset\cL(H_{\infty,\semi}^{s,r,p},
H_{\infty,\semi}^{s-m,r-l,p-k}).
$$
Moreover, $H_\infty^{s,p}=H_{\infty,\semi}^{s,s+p,p}$ (with the left
hand side defined using $\Psi_{\infty}^{s,p}$ in place of
$\Psi_{\infty,\semi}^{s,r,p}$, so this is an immediate consequence of
the definitons and the relationship between the pseudodifferential algebras), and
$H_\semi^{s,r}=H_{\infty,\semi}^{s,r,r}$ (with the left hand side
defined using the semiclassical algebra); this different character of
the identification of the classical and semiclassical Sobolev spaces
as a joint space is what makes mapping properties of semiclassical
operators on classical spaces more subtle.

Now, the typical issue with semiclassical estimates on
non-semiclassical spaces is that the orders of the operator do not
conform to the standard case. For instance, the spectral family
$\Delta-z$, where $z$ runs to infinity in a cone disjoint from the
positive reals, can be written as $h^{-2}(h^2\Delta-\lambda)$, so
$z=\lambda/h^2$, and now $\lambda$ is bounded away from $[0,\infty)$
and is bounded. As $\Delta\in\Psi_\infty^{2,0}$,
$h^{-2}\lambda\in\Psi_\infty^{0,2}$, we have
$$
h^{-2}(h^2\Delta-\lambda)\in \Psi_{\infty,\semi}^{2,2,2}$$
and it is elliptic in this
class since its principal symbol is
$h^{-2}(|\zeta_{\semi}|^2_g-\lambda)$, where the second term is only
relevant at the front face and it being bounded away from $[0,\infty)$
assures the appropriate lower bound for this principal
symbol. Moreover, its normal operator is $h^{-2}\lambda \in \Psi_{\infty,\semi}^{0,2,2}$, i.e.\ a non-zero
multiple of the identity, which is certainly
invertible. Correspondingly $h^{-2}(h^2\Delta-\lambda)$ is fully
elliptic, and thus is invertible for small $h$ with
$$
\big(h^{-2}(h^2\Delta-\lambda)\big)^{-1}\in
\Psi_{\infty,\semi}^{-2,-2,-2}.
$$
Hence,
$$
(\Delta-z)^{-1}\in\cL(H_{\infty,\semi}^{s,r,p},
H_{\infty,\semi}^{s+2,r+2,p+2}).
$$
Of course, by \eqref{eq:semi-to-combined-conv}, in this case we have the stronger statement
$$
h^{-2}(h^2\Delta-\lambda)\in \Psi_\semi^{2,2}\subset\Psi_{\infty,\semi}^{2,2,2},
$$
is elliptic in this semiclassical algebra, hence the inverse lies in
$$
\big(h^{-2}(h^2\Delta-\lambda)\big)^{-1}\in
\Psi_{\semi}^{-2,-2}\subset\Psi_{\infty,\semi}^{-2,-2,-2},
$$
and then we can simply use the final inclusion here to obtain the
mapping properties above. However, as we point out below, proceeding
with the joint algebra from the start immediately extends the argument
to the spectral family of {\em pseudodifferential} operators.

Now, if we take $r=s+p$ so that the domain is $H_\infty^{s,p}$, then
the output is in $H_{\infty,\semi}^{s+2,s+p+2,p+2}$, but this is not
$H_\infty^{s+2,p+2}$ as the first and last orders may suggest!
Correspondingly, in order to work with purely non-semiclassical spaces
one has to give up something and have, with $t\in[0,2]$,
\begin{equation}\label{eq:interpolate-mapping}
(\Delta-z)^{-1}\in\cL(H_{\infty}^{s,p},
H_{\infty}^{s+t,p+2-t}).
\end{equation}
Since the output space is $H_{\infty,\semi}^{s+t,s+p+2,p+2-t}$, this
is actually sharp in the second order sense, but one is juggling
whether to give up differentiability or decay in the standard sense as
$h\to 0$. Note that the extreme cases, which immediately imply \eqref{eq:interpolate-mapping}, are the well-known
$$
(\Delta-z)^{-1}\in\cL(H_{\infty}^{s,p},
H_{\infty}^{s,p+2})\cap\cL(H_{\infty}^{s,p},
H_{\infty}^{s+2,p});
$$
the former embodies the large parameter decay, the second standard ellipticity,
but in any case there is a compromise.

Taking advantage of the joint standard/semiclassical algebra we can
similarly prove Theorem~\ref{theorem:main}:

\begin{proof}[Proof of Theorem~\ref{theorem:main}]
The large parameter spectral family for pseudodifferential operators
$A\in\Psi_\infty^{m,0}$, $m>0$, is included in the joint
standard/semiclassical algebra, namely with $z=\lambda/h^m\in\Psi_{\infty,\semi}^{0,m,m}$,
$$
A-z=h^{-m}(h^m A-\lambda)\in\Psi_{\infty,\semi}^{m,m,m}.
$$
Thus, completely analogous arguments as for the Laplacian apply (with $2$ replaced by
$m$). In particular, if $A$ has a real elliptic symbol and $\lambda$ is
in a compact set disjoint from $\RR$ (and either $h$ is small or $A$
is self-adjoint), we have
$$
(A-z)^{-1}\in \Psi_{\infty,\semi}^{-m,-m,-m}=\Psi_\infty^{-m,0}\cap\Psi_\infty^{0,-m}.
$$
All of these arguments go through in the bounded geometry
setting. They also go through in other operator algebras like the
scattering algebra.
\end{proof}

\bibliographystyle{amsplain}

\bibliography{sm}

\end{document}